\documentclass[a4paper,11pt]{amsproc}

\usepackage{amssymb,latexsym, amsmath}
\usepackage[dvips]{graphicx}
\usepackage{amscd}

\newtheorem{remark}{Remark}
\newtheorem{theorem}{Theorem}
\newtheorem{proposition}{Proposition}

\newtheorem{definition}{Definition}

\newtheorem{lemma}{Lemma}

\newcommand{\R}{\mathbb{R}}

\DeclareMathOperator{\orth}{\mathrm{orth}}

\begin{document}

\title[NONCOMMUTATIVE INTEGRABILITY  IN CONTACT
GEOMETRY ]{NONCOMMUTATIVE INTEGRABILITY AND ACTION-ANGLE VARIABLES
IN CONTACT GEOMETRY}


\author[B. Jovanovi\'c]{\bfseries Bo\v zidar Jovanovi\'c}

\address{
Mathematical Institute \\
Serbian Academy of Sciences and Arts \\
Kneza Mihaila 36, 11000 Belgrade\\
Serbia} \email{bozaj@mi.sanu.ac.rs}


\maketitle

\begin{abstract}
We introduce a notion of the noncommutative integrability within a
framework of contact geometry.
\end{abstract}

\section{Introduction}

\subsection{}
In Hamiltonian mechanics solving by quadratures is closely related
to the regularity of dynamics that is described in the
Arnold-Liouville theorem. A Hamiltonian system  on
$2n$-dimensional symplectic manifold $M$ is called {\it
integrable} if it has $n$ smooth Poisson-commuting, almost
everywhere independent integrals $f_1,f_2,\dots,f_n$.  Regular
compact connected invariant manifolds of the system are Lagrangian
tori. Moreover, in a neighborhood of any torus, there exist
canonical action-angle coordinates
$(\varphi,I)=(\varphi_1,\dots,\varphi_n,I_1,\dots,I_n)$, integrals
$f_i$ depend only on actions $I$ and the flow is translation in
$\varphi$ coordinates \cite{Ar}.

Therefore, an integrable Hamiltonian system can be considered as a
toric Lagrangian fibration $ \pi: M\to W$ (see Duistermaat
\cite{Du}). This approach is reformulated  to contact manifolds
$(M,\mathcal H)$ by Banyaga and Molino \cite{BM}. Instead of a
toric Lagrangian fibration, one consider an invariant toric
fibration transversal to the contact distribution $\mathcal H$,
such that intersection of tori and $\mathcal H$ is a Lagrangian
distribution with respect to the conformal class of the symplectic
structure on $\mathcal H$ (see Section 5).

Slightly different notion of a contact integrability is given
recently by  Khesin and Tabachnikov \cite{KT}. They defined
integrability in terms of the existence of an invariant foliation
$\mathcal F$, called a co-Legendrian foliation (here we refer to
$\mathcal F$ as a pre-Legendrian foliation). $\mathcal F$ is
transversal to $\mathcal H$, $\mathcal G=\mathcal F\cap \mathcal
H$ is a Legendrian foliation of $M$ with an additional property
that on every leaf $F$ of $\mathcal F$, the foliation $\mathcal
G\vert_F$ has a holonomy invariant transverse smooth
measure.\footnote{Through the paper we use the same notation for
foliations and their integrable distributions of tangent spaces.}
It turns out that this condition implies the existence of a global
contact form $\alpha$ (see \cite{KT}) and $\mathcal G$ is a
$\alpha$-complete Legendrian foliation studied by Libermann
\cite{L2} and Pang \cite{P}. Recall that a foliation $\mathcal F$
is {\it $\alpha$-complete} if for any pair $f_1,f_2$ of first
integrals of $\mathcal F$ (where $f_i$ may be a constant), the
Jacobi bracket $[f_1,f_2]$ is also a first integral of $\mathcal
F$ (eventually a constant).

Due to the presence of symmetries, many Hamiltonian systems have
more than $n$ non-commuting integrals. Illustrative examples are
$G$-invariant geodesic flows on homogeneous spaces \cite{BJ, Jo1}.
An appropriate framework for the study of these systems is
noncommutative integrability introduced by Nehoroshev \cite{N} and
Mishchenko and Fomenko \cite{MF} (see also \cite{BJ, Zu, Jo1,
LMV}). Here we recall the Nehoroshev formulation: a Hamiltonian
system on $2n$-dimensional symplectic manifold $M$ is {\it
noncommutatively integrable} if it has $2n-r$ almost everywhere
independent integrals $f_1,f_2,\dots,f_{2n-r}$ and $f_1,\dots,f_r$
commute with all integrals
$$
\{f_i,f_j\}=0, \quad i=1,\dots,2n-r, \quad j=1,\dots,r.
$$
Regular compact connected invariant manifolds of the system are
isotropic tori. In a neighborhood of a regular torus, there exist
canonical {\it generalized action-angle coordinates} such that
integrals $f_i$, $i=1,\dots,r$ depend only on actions and the flow
is translation in angle coordinates.

One of the basic examples of contact manifolds are unit co-sphere
bundles $SQ\subset T^*Q$ of Riemannian manifolds $(Q,g)$. The
restriction of a geodesic flow to $SQ$ is a contact flow of the
Reeb vector field of the associated contact form. It is clear that
noncommutatively integrable geodesic flows, considered as Reeb
vector flows, have a geometrical structure that need to be
described by a noncommutative variant of integrability.

\subsection{}
We introduce an appropriate concept of a contact noncommutative
integrability.

In the first part of the paper (sections 3 and 4) foliations on
contact manifolds $(M,\mathcal H)$ are considered. We refer to a
foliation $\mathcal F$ as  {\it pre-isotropic} if it is
transversal to $\mathcal H$ and $\mathcal G=\mathcal F\cap\mathcal
H$ is an isotropic subbundle of $\mathcal H$.

Let $\mathcal F$ be a pre-isotropic foliation containing the Reeb
vector field $Z$ on a co-oriented contact manifold $(M,\alpha)$.
The foliation $\mathcal F$ is $\alpha$-complete if and only if
$\mathcal E$ is completely integrable, where $\mathcal E=\mathcal
F^\perp$ is the {\it pseudo-orthogonal distribution} of $\mathcal
F$ and we have a flag of foliations $\mathcal G \subset \mathcal F
\subset\mathcal E$. Furthermore, each leaf of $\mathcal G$ and
$\mathcal F$ has an affine structure (Theorem 2).

Thus, if $\mathcal F$ has compact leaves, they are tori. Locally,
in an invariant neighborhood of any leaf, the foliation $\mathcal
F$ can be seen as a fibration over some base manifold. Also,
affine translations provide an Abelian Lie algebra of contact
transformations with orbits that coincide with $\mathcal F$.

Next, we consider a pre-isotropic foliation $\mathcal F$ on a
contact manifold $(M,\mathcal H)$ with the mentioned properties of
$\alpha$-complete pre-isotropic foliations: $\mathcal F$ is
defined via submersion $\pi: M\to W$ and it is given an Abelian
Lie algebra of contact symmetries $\mathcal X$ with orbits equal
to $\mathcal F$. We refer to a triple $(M,\mathcal H,\mathcal X)$
as a {\it complete pre-isotropic contact structure}.

For a given complete pre-isotropic structure $(M,\mathcal
H,\mathcal X)$, locally, there always exist an invariant contact
form $\alpha$ such that $\mathcal F$ is $\alpha$-complete (Theorem
3). Notice, if $\mathcal F$ has the maximal dimension (i.e, it is
pre-Legendrian) and fibers of $\pi$ are connected, $(M,\mathcal
H,\mathcal X)$ is a {\it regular completely integrable contact
structure} studied by Banyaga and Molino \cite{BM}. The analysis
above lead us to the following definition (section 5).

Let $X$ be a contact vector field. We shall say that a contact
equation
\begin{equation}\label{CE}
 \dot x=X
 \end{equation}
 is {\it contact noncommutatively integrable} if there is an Abelian Lie algebra of contact symmetries
$\mathcal X$, an open dense set $M_{reg}\subset M$, and a
submersion $\pi:M_{reg}\to W$ such that
\begin{itemize}
\item[(i)] the contact vector field $X$ is tangent to the fibers
of $\pi$;
\item[(ii)] $(M_{reg},\mathcal H,\mathcal X)$ is a
complete pre-isotropic contact structure.
\end{itemize}

Analogues to the Mishchenko--Fomenko--Nehoroshev theorem, we prove
that in a neighborhood of any invariant torus, there exist
canonical generalized contact action-angle coordinates and
\eqref{CE} is a translation in angle variables, where frequencies
depend only on actions and in which the contact distribution
$\mathcal H$ is presented by the canonical 1-form $\alpha_0$
(Theorem 4).

For the co-oriented case, we also formulate the statement
involving only integrals of a motion (Theorem 5): a contact
equation \eqref{CE} is noncommutatively integrable if it possesses
a collection of first integrals $f_1,f_2,\dots,f_{2n-r}$, that are
all in involution with the constant functions and with the first
$r$ integrals:
\begin{equation*}
[1,f_i]=0, \quad [f_i,f_j]=0, \quad i=1,\dots,2n-r, \quad
j=1,\dots, r.
\end{equation*}

Note that,  besides integrable geodesic flows on homogeneous
spaces restricted to the unit co-sphere bundles \cite{BJ, Jo1}, a
natural class of examples of contact flows integrable in a
noncommutative sense are the Reeb flows on $K$-contact manifolds
$(M^{2n+1},\alpha)$ where the rank of the manifold is less then
$n+1$ (see Yamazaki \cite{Y1} and Lerman \cite{Le2}).

Finally in section 6, we consider a complete pre-isotropic contact
structure $(M,\alpha,\mathcal X)$ of the {\it Reeb type}, that is
$\mathcal H$ is defined by a global $\mathcal X$-invariant form
$\alpha$ and the Reeb vector field of $\alpha$ is $\pi$-vertical.
Note that $(M,\mathcal H,\mathcal X)$ can be a complete
pre-isotropic structure with a global $\mathcal X$-invariant
contact form $\alpha$, which is not of the Reeb type (see
Proposition 1). On the other hand, the invariant foliation
$\mathcal F$ of a complete pre-isotropic contact structure
$(M,\alpha,\mathcal X)$ of the Reeb type is $\alpha$-complete
(Proposition 2). We describe the transition functions between the
contact action-angles coordinates (Proposition 3) and prove the
statement on the existence of global action-action variables in
the case when $\pi: M\to W$ is a trivial principal  $\mathbb
T^{r+1}$-bundle (Theorem 6).

\section{Contact manifolds and the Jacobi bracket}

\subsection{}
In the definitions and notations we mostly follow Libermann and
Marle \cite{LM}.

A {\it contact form} $\alpha$ on a $(2n+1)$-dimensional manifold
$M$ is a Pfaffian form satisfying $\alpha\wedge(d\alpha)^n\ne0$.
By a {\it contact manifold} $(M,\mathcal H)$ we mean a connected
$(2n+1)$-dimensional manifold $M$ equipped with a nonintegrable
{\it contact} (or {\it horizontal}) {\it distribution} $\mathcal
H$, locally defined by a contact form: $\mathcal
H\vert_U=\ker\alpha\vert_U$, $U$ is an open set in $M$.

Two contact forms $\alpha$ and $\alpha'$ define the same contact
distribution $\mathcal H$ on $U$ if and only if $\alpha'=a\alpha$
for some nowhere vanishing function $a$ on $U$. The condition
$\alpha\wedge(d\alpha)^n\ne0$ implies that the form
$d\alpha\vert_x$ is nondegenerate (symplectic) structure
restricted to $\mathcal H_x$. The conformal class of
$d\alpha\vert_x$ is invariant under the change $\alpha'=a\alpha$.
If $\mathcal V$ is a linear subspace of $\mathcal H_x$, then we
have well defined orthogonal complement $\orth_\mathcal H \mathcal
V\subset \mathcal H_x$ with respect to $d\alpha\vert_x$, as well
as the notion of the {\it isotropic} ($\mathcal V \subset
\orth_\mathcal H \mathcal V$), {\it coisotropic} ($\mathcal V
\supset \orth_\mathcal H \mathcal V$) and the {\it Lagrange}
subspaces ($\mathcal V =\orth_\mathcal H \mathcal V$) of $\mathcal
H_x$ .

A {\it contact diffeomorphism} between contact manifolds
$(M,\mathcal H)$ and $(M',\mathcal H')$ is a diffeomorphism $\phi:
M\to M'$ such that $\phi_*\mathcal H=\mathcal H'$. If a local
1-parameter group of a vector field $X$ is made of contact
diffeomorphisms, $X$ is called an {\it infinitesimal automorphism}
of a contact structure $(M,\mathcal H)$ or  a {\it contact vector
field}. Locally, if $\mathcal H=\ker\alpha$, then $\mathcal
L_X\alpha=\lambda\alpha$, for some smooth function $\lambda$.

The existence of a global contact form $\alpha$ is equivalent to
the coorientability of $\mathcal H$ \cite{Gr}.   From now on we
consider a {\it co-oriented (or stricly) contact manifold}
$(M,\alpha)$. The {\it Reeb vector field} $Z$ is a vector field
uniquely defined by
$$
i_Z\alpha=1, \qquad i_Z d\alpha=0.
$$

The tangent bundle $TM$ and the cotangent bundle $T^*M$ are
decomposed into
\begin{equation}
TM=\mathcal Z \oplus \mathcal H, \qquad T^*M=\mathcal Z^0\oplus
\mathcal H^0, \label{decomposition}
\end{equation}
 where $\mathcal Z=\R Z$ is the kernel of
$d\alpha$, $\mathcal Z^0$ and $\mathcal H^0=\R\alpha$ are the
annihilators of $\mathcal Z$ and $\mathcal H$, respectively. The
sections of $\mathcal Z^0$ are called {\it semi-basic forms}.

According to \eqref{decomposition}, we have decompositions of
vector fields and 1-forms
\begin{equation}
X=(i_X\alpha)Z+\hat X, \qquad \eta=(i_Z\eta)\alpha+\hat \eta,
\label{decomp*}
\end{equation}
where $\hat X$ is horizontal and $\hat\eta$ is semi-basic.

The mapping $\alpha^\flat: X \mapsto -i_X d\alpha$ carries $X$
onto a semi-basic form. The restriction of $\alpha^\flat$ to
horizontal vector fields is an isomorphism whose inverse will be
denoted by $\alpha^\sharp$. The mapping
\begin{equation}
 \Phi: \,\mathcal N \longrightarrow C^\infty(M), \qquad  \Phi(X)=i_X\alpha \label{iso}
\end{equation}
establish the isomorphism  between the vector space $\mathcal N$
of infinitesimal contact automorphisms onto the set $C^\infty(M)$
of smooth functions on $M$, with the inverse (see \cite{L0, LM})
$$
\Phi^{-1}(f)=fZ+\alpha^\sharp(\widehat{df}).
$$

The vector field $X_f=\Phi^{-1}(f)$ is called the {\it contact
Hamiltonian vector field} and
\begin{equation}\label{eq_f}
\dot x=X_f
\end{equation}
{\it contact Hamiltonian equation} corresponding to $f$. Note that
$$
\mathcal L_{X_f}\alpha=df(Z)\alpha
$$
and $X_f$ is an {\it infinitesimal automorphism of $\alpha$}
($\mathcal L_{X_f}\alpha=0$) if and only if $df$ is semi-basic.
Notice that $\Phi(Z)=1$, i.e., $Z=X_1$.

\subsection{} The mapping \eqref{iso} is a Lie algebra isomorphism, where on
$\mathcal N$ we have the usual bracket and the {\it Jacobi
bracket} on $C^\infty(M)$ defined by $[f,g]=\Phi[X_f,X_g]$:
\begin{equation*}
X_{[f,g]}=[X_f,X_g], \qquad X_{[1,f]}=[Z,X_f].
\end{equation*}
Note that $df$ is sami-basic if and only if $[1,f]=[Z,X_f]=0$.

Together with the Jacobi bracket, we have  the associated {\it
Jacobi bi-vector field} $\Lambda$:
$$
\Lambda(\eta,\xi)=d\alpha(\alpha^\sharp\hat\eta,\alpha^\sharp\hat\xi).
$$

Let $\Lambda^\sharp: T^*M \to TM$ be the morphism defined by
$\langle\Lambda^\sharp_x(\eta_x),\xi_x\rangle=\Lambda_x(\eta_x,\xi_x)$,
for all $x\in M$, $\eta_x,\xi_x\in T^*_xM$. Then
 $X_f$
may be written as $X_f=fZ+\Lambda^\sharp(df)$.

It can be easily checked that
$$
[f,g]=d\alpha(X_f,X_g)+f\mathcal L_Z g-g\mathcal L_Z
f=\Lambda(df,dg)+f\mathcal L_Zg-g\mathcal L_Z f.
$$

The derivation of functions along the contact vector field $X_f$
can be described by the use of the Jacobi bracket
\begin{equation}\label{der}
\mathcal L_{X_f} g=[f,g]+g\mathcal L_Z f.
\end{equation}

Thus, if $df$ and $dg$ are semi-basic, we have the following
important property of the Jacobi bracket $[f,g]$.

\begin{lemma}\label{SB}
Suppose that $df$ and $dg$ are semi-basic. Then
$$[f,g]=d\alpha(X_f,X_g)=\Lambda(df,dg)$$ and the following
statements are equivalent:

(i) $f$ and $g$ are in involution: $[f,g]=0$, i.e., Hamiltonian
contact vector fields $X_f$ and $X_g$ commute: $[X_f,X_g]=0$.

(ii) $g$ is the integral of the contact vector field $X_f$:
$\mathcal L_{X_f} g=0$.

(iii) $f$ is the integral of the contact vector field $X_g$:
$\mathcal L_{X_g} f=0$.
\end{lemma}

Moreover if $\mathcal Z$ is a {simple foliation}, that is, there
exist a surjective submersion $\pi: M\to P$ and the distribution
$\mathcal Z$ consist of vertical spaces of the submersion:
$\mathcal Z=\ker\pi_*$, then the base manifold $P$ has a
non-degenerate Poisson structure $\{\cdot,\cdot\}$ such that
$[f,g]=\pi^*\{\bar f,\bar g\}$, $f=\bar f\circ\pi$, $g=\bar
g\circ\pi$ \cite{BW, LM}.

\section{$\alpha$-Complete pre-isotropic foliations}

Let $\mathcal F$ be a foliation on a co-oriented contact manifold
$(M^{2n+1},\alpha)$. The {\it pseudo-orthogonal distribution}
$\mathcal F^\perp$ is defined by
$$
\mathcal F^\perp=\mathcal Z\oplus\Lambda^\sharp(\mathcal F^0).
$$
where $\mathcal F^0$ is the annihilator of $\mathcal F$. It is
locally generated by the Reeb vector field $Z$ and the contact
Hamiltonian vector fields which corresponds to the first integrals
of $\mathcal F$.

A foliation $\mathcal F$ is said to be {\it $\alpha$-complete} if
for any pair $f_1,f_2$ of first integrals  of $\mathcal F$ (where
$f_i$ may be a constant), the bracket $[f_1,f_2]$ is also a first
integral of $\mathcal F$ (eventually a constant).

\begin{theorem}[Libermann \cite{L1}]
A foliation $\mathcal F$ on $(M^{2n+1},\alpha)$ containing the
Reeb vector field $Z$ is $\alpha$-complete if and only if the
pseudo-orthogonal subbundle $\mathcal F^\perp$ is integrable,
defining a foliation which is also $\alpha$-complete and
$(\mathcal F^\perp)^\perp=\mathcal F$. Then for any pair of
integrals $f,g$ of $\mathcal F$ and $\mathcal F^\perp$,
respectively, we have $[f,g]=0$.
\end{theorem}

Let $p$ be the rank of $\mathcal F^0$ and $f_1,\dots,f_p$ be a set
of independent integrals of $\mathcal F$ in an open set $U$. Since
$\ker\Lambda_x^\sharp=\mathbb R\alpha_x$, $\dim \mathcal
F^\perp_x$ is equal to $p+1$ or $p$, depending the forms
$\alpha,df_1,\dots,df_p$ are linearly independent or not. In the
later case, the form induced by $\alpha$ on the leaf passing
through $x$ vanishes at $x$. Conversely, if $\alpha\vert_{\mathcal
F}=0$, i.e., $\mathcal F\subset\mathcal H$, then $\dim\mathcal
F^\perp=p$.

A foliation $\mathcal G$ is {\it pseudo-isotropic} if $\mathcal
G\subset \mathcal H$ \cite{L2}.\footnote{Submanifolds $G\subset M$
that are integral manifolds of $\mathcal H$ are also called
isotropic submanifolds, e.g., see \cite{Ge}. Here we keep
Libermann's notation.} Then $\alpha$ is a section of $\mathcal
G^0$, the distribution $\mathcal G^\perp$ has the constant rank
$p$ and $\mathcal G^\perp$ is a vector bundle. A {\it Legendre
foliation} is a pseudo-isotropic foliation of maximum rank $n$.
Then $\dim\mathcal G^0=\dim\mathcal G^\perp=n+1$.

By the analogy with a pre-isotropic embedding (see Lerman
\cite{Le}), we introduce:

\begin{definition}{\rm
A foliation $\mathcal F$ is {\it pre-isotropic} if

\begin{itemize}

\item[(i)] $\mathcal F$ is transverzal to $\mathcal H$

\item[(ii)] $\mathcal G=\mathcal F \cap\mathcal H$ is an isotropic
subbundle of $\mathcal H$.

\end{itemize}
}\end{definition}

\begin{lemma}\label{fol}
The condition (ii) is equivalent to the condition that $\mathcal
G=\mathcal F\cap \mathcal H$ is a pseudo-isotropic foliation.
\end{lemma}

\begin{proof} Let $(f_1,\dots,f_p)$ be a set of local integrals of
$\mathcal F$ and let $X,Y$ be sections of $\mathcal G$. Then
$\alpha$, $df_1,\dots,df_p$ are linearly independent and we have
\begin{eqnarray*}
&&df_i(X)=df_i(Y)=\alpha(X)=\alpha(Y)=0,\\
&&df_i([X,Y])=\mathcal L_X\mathcal L_Y f_i-\mathcal L_Y\mathcal
L_X f_i=0,\\
&&d\alpha(X,Y)=\mathcal L_X\alpha(Y)-\mathcal
L_Y\alpha(X)-\alpha([X,Y])=-\alpha([X,Y]).
\end{eqnarray*}
Therefore $\mathcal G$ is an isotropic subbundle of $\mathcal H$
if and only if it is integrable. \end{proof}

\begin{theorem} Let $\mathcal F$ be a pre-isotropic foliation
containing the Reeb vector field $Z$.

(i)  We have the flag of distributions $(\mathcal G,\mathcal
F,\mathcal E)$:
\begin{equation}\label{flag}
\mathcal G=\mathcal F \cap \mathcal H \,\,\subset\,\, \mathcal F
\,\,\subset \,\,\mathcal E=\mathcal G^\perp=\mathcal F^\perp.
\end{equation}

Contrary, if $\mathcal F$ is a  foliation containing the Reeb
vector field $Z$ and \eqref{flag} holds, then $\mathcal F$ is a
pre-isotropic foliation.

(ii) The foliation $\mathcal F$ (or $\mathcal G$) is
$\alpha$-complete if and only if $\mathcal E$ is completely
integrable. Assume $\mathcal E$ is integrable and let
$f_1,\dots,f_p$ and $y_1,\dots,y_r$, $2n-p=r$ be any sets of local
integrals of $\mathcal F$ and $\mathcal E$, respectively. Then:
$$
[f_i,y_j]=0, \quad [y_j,y_k]=0, \quad [f_i,1]=0, \quad [y_i,1]=0.
$$

(iii) Each leaf of an $\alpha$-complete pre-isotropic foliation
$\mathcal F$ as well as each leaf of the corresponding
pseudo-isotropic foliation $\mathcal G$ has an affine structure.
\end{theorem}

\begin{figure}[ht]
\begin{center}
\includegraphics[scale=3]{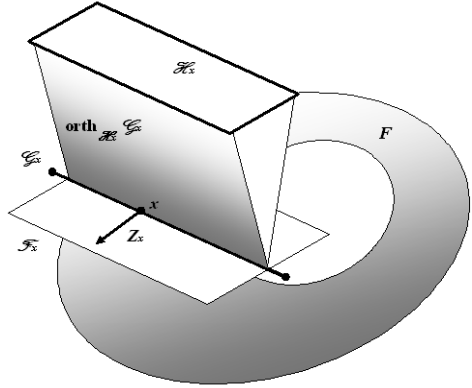}
\caption{Illustration of Theorem 2: a torus $F$ is a leaf through
$x$.}
\end{center}
\end{figure}

\begin{proof} (i) Let $\mathcal G=\mathcal F\cap \mathcal H$ be isotropic.
We have $\mathcal G^0=\langle \mathcal F^0,\alpha\rangle$ and
$\ker \Lambda^\sharp=\R\alpha$. Thus:
\begin{eqnarray*}
\mathcal F^\perp=\mathcal G^\perp &=& \mathcal
Z\oplus\Lambda^\sharp(\mathcal G^0) =\mathcal Z\oplus
\Lambda^\sharp(\mathcal G^0 \cap\mathcal Z^0)\\
&=&\mathcal Z\oplus \alpha^\sharp(\mathcal G^0 \cap \mathcal
Z^0)=\mathcal Z\oplus \orth_\mathcal H\mathcal G \supset \mathcal
Z \oplus \mathcal G =\mathcal F.
\end{eqnarray*}

\medskip

{(ii)} This item follows directly from Theorem 1 and the fact that
integrals of $\mathcal E$ are also integrals of $\mathcal F$.
Notice that $df_i$ and $dy_j$ are semi-basic and $X_{f_i}$,
$X_{y_j}$ are infinitesimal automorphisms of $\alpha$.

\medskip

{(iii)} Let $U$ be an open  set where we have defined commuting
integrals $y_1,\dots,y_r$ of $\mathcal E\vert_U$. Since $\mathcal
E^\perp=\mathcal F$, the distribution $\mathcal F\vert_U$ is
generated by a contact commuting vector fields
$Z,X_{y_1},\dots,X_{y_r}$:
$$
[Z, X_{y_i}]=0, \qquad [X_{y_i},X_{y_j}]=0.
$$

The distribution $\mathcal G\vert_U$ is generated by their
horizontal parts $\hat X_{y_1},\dots,\hat X_{y_r}$ which also
commute. Indeed, since $\mathcal G$ is integrable $[\hat
X_{y_i},\hat X_{y_j}]$ is a section of $\mathcal G$, in particular
it is horizontal. Further
\begin{eqnarray}
0&=& [X_{y_i},X_{y_j}]=[y_i Z+\hat X_{y_i},y_j Z+\hat X_{y_j}]
\nonumber\\
&=& [y_i Z,y_jZ]+[\hat X_{y_i},\hat X_{y_j}] +y_i[Z,\hat
X_{y_j}]+y_j[\hat X_{y_i},Z] \label{prva}\\
&& -\mathcal L_{\hat X_{y_j}}(y_i) Z+\mathcal L_{\hat X_{y_i}}
(y_j) Z.\nonumber
\end{eqnarray}
On the other side, since $\mathcal L_Z y_i=0$, we have
\begin{eqnarray}
0&=&[Z,X_{y_i}]=[Z,y_iZ+\hat X_{y_i}]=[Z,y_i Z]+[Z,\hat
X_{y_i}]\nonumber \\
&=& \mathcal L_Z(y_i)Z+[Z,\hat X_{y_i}]=[Z,\hat X_{y_i}].
\label{druga}
\end{eqnarray}
Therefore, taking the horizontal part in \eqref{prva} we get
$$
[\hat X_{y_i},\hat X_{y_j}]=0.
$$

Thus, locally we have parallelism both on $\mathcal F=\langle Z,
\hat X_{y_1},\dots,\hat X_{y_r}\rangle$ and $\mathcal G=\langle
\hat X_{y_1},\dots,\hat X_{y_r}\rangle$. Now, let $U'$ be an open
set ($U\cap U'\ne\emptyset$) and let $y_1',\dots,y_r'$ be
commuting integrals of $\mathcal E\vert_{U'}$. Then, on $U\cap U'$
we have
\begin{eqnarray*}
&&y_i'=\varphi_i(y_1,\dots,y_r), \qquad i=1,\dots,r\\
&&dy_i'=\sum_j \frac{\partial \varphi_i}{\partial y_j}dy_j.
\end{eqnarray*}

From the definition $\hat
X_{y_i'}=\alpha^\sharp(\widehat{dy'_i})=\alpha^\sharp(dy_i'-(i_Z
dy_i')\alpha)=\alpha^\sharp(dy'_i)$, we get the fiber-wise linear
transformation
$$
\hat X_{y_i'}=\sum_j \frac{\partial \varphi_i}{\partial y_j} \hat
X_{y_j'}, \qquad i=1,\dots,r
$$
which shows that the parallelism of $\mathcal G$ and $\mathcal F$
is independent of the chart. \end{proof}

If $\mathcal F$ has the maximal dimension $n+1$ then $\mathcal F$
is  pre-Legendrian, while $\mathcal G$ is a Legendrian foliation.
The existence of an affine structure is already known for
$\alpha$-complete Legendre foliations \cite{L1, L2,P, KT}. This
imposes restrictions on the topology of the leaves. In particular,
compact leaves of $\mathcal G$ and $\mathcal F$ are tori.

Of particular interest is the case when $\mathcal F$ is a simple
foliation, i.e., the leaves of the foliation are fibers of the
submersion. We will study such a situation in the next section.

\section{Complete pre-isotropic contact structures}

In this section, a contact structure does not need to be
co-oriented.

Let $(M,\mathcal H)$ be a $(2n+1)$-dimensional contact manifold
and let
\begin{equation}
\pi: M \to W \label{submersion}
\end{equation}
be a proper submersion on $p$-dimensional manifold $W$, $p\ge n$.
Define the distribution $\mathcal F$ as the kernel of $\pi_*: TM
\to TW$, i.e., the leaves of $\mathcal F$ are fibers of $\pi$.

\begin{definition}{\rm
We shall say that $(M,\mathcal H,\mathcal X)$ is a {\it complete
pre-isotropic contact structure} if
\begin{itemize}
\item[(i)] $\mathcal F$ is pre-isotropic, i.e., it is transversal
to $\mathcal H$ and $\mathcal G=\mathcal F\cap \mathcal H$ is an
isotropic subbundle of $\mathcal H$, or, equivalently $\mathcal G$
is a foliation;

\item[(ii)] $\mathcal X$ is an Abelian Lie algebra of
infinitesimal contact automorphisms of $\mathcal H$, which has the
fibers of $\pi$ as orbits.
\end{itemize}
 }\end{definition}

In the case $p=n$ (and connected fibers) we have a {\it regular
completely integrable contact structure} $(M,\mathcal H,\mathcal
X)$ studied in Banyaga and Molino \cite{BM}.

Suppose $\mathcal F$ is an $\alpha$-complete foliation with
compact leaves (according to the presence of the affine structure,
the leaves are tori). Locally, in a neighborhood $U$ of any fixed
torus $F$ the foliation is simple. There is a surjective
submersion $\pi: U\to W=U/\mathcal F$, $\mathcal F=\ker\pi_*$. We
can define an Abelian Lie algebra $\mathcal X$ of infinitesimal
automorphisms of $\mathcal H$ by $Z, X_{y_1},\dots,X_{y_r}$ where,
$y_1,\dots,y_r$ are integrals of $\mathcal E=\mathcal F^\perp$.
Thus, we have well defined complete pre-isotropic contact
structure $(U,\mathcal H,\mathcal X)$.

Contrary, we have also:

\begin{theorem}
Let $(M,\mathcal H,\mathcal X)$ be a complete pre-isotropic
contact structure related to the submersion \eqref{submersion}.
Every point of $M$ has an open, $\mathcal X$-invariant
neighborhood $U$ on which the contact structure can be represented
by a local contact form $\alpha_U$ such that:

(i) $\alpha_U$ is invariant by all elements of $\mathcal X$;

(ii) the restriction of $\mathcal F$ to $U$ is
$\alpha_U$-complete.
\end{theorem}

\begin{proof} (i) The proof of item (i) is a modification of the proof
given in \cite{BM} for a regular completely integrable contact
structure. From the definition, for every point $x_0$ of $M$,
there exist $X\in\mathcal X$ transverze to $\mathcal H_{x_0}$. The
vector field $X$ is then transvrerze to $\mathcal H$ in some
neighborhood $U_2$ of $x_0$. Let $\alpha_0$ be a contact form
defining $\mathcal H$ in $U_1\subset U_2$. Then $\alpha_0(X)\ne 0$
on $U_1$ and define $\alpha=\alpha_0/\alpha_0(X)$.

Since $\mathcal X$ is Abelian, we have  $i_{[Y,X]}\alpha=0$,
$Y\in\mathcal X$.  Also, $i_X\alpha=1$ and $\mathcal
L_Y\alpha=\lambda\alpha$, for some function $\lambda$ defined in
$U_1$. Thus
$$
0=i_{[Y,X]}\alpha=\mathcal L_Y i_X \alpha - i_X \mathcal L_Y
\alpha=\mathcal L_Y 1-i_X(\lambda\alpha)=-\lambda,
$$
i.e., $Y$ is an infinitesimal automorphism of $\alpha$. Since
$\alpha$ is invariant by $\mathcal X$ and the orbits of $\mathcal
X$ are the fibers of the submersion \eqref{submersion}, the form
$\alpha$ is well defined on $U=\pi^{-1}(\pi(U_1))$ as well.

\medskip

{(ii)} The foliation $\mathcal F\vert_U$ is $\alpha$-complete if
and only if $\mathcal E\vert_U=\mathcal F^\perp\vert_U$ is an
integrable distribution.

From the identity
$$
0=\mathcal L_X \alpha=i_X d\alpha+ di_X\alpha=i_Xd\alpha
$$
we get that $X$ is the Reeb vector field of $\alpha$ on $U$.
Denote $Z=X$.

Let $X_1,\dots,X_r\in\mathcal X$ be vector fields such that
$Z,X_1,\dots,X_r$ span the foliation $\mathcal F\vert_U$.
Therefore, the corresponding contact Hamiltonians
$$
y_i=\Phi(X_i)=i_{X_i}\alpha
$$
are independent functions on $U$. Besides, $y_i$ are
$\pi$-vertical:
$$
0=i_{[X,X_i]}\alpha=\mathcal L_X i_{X_i} \alpha - i_{X_i} \mathcal
L_X \alpha=\mathcal L_X y_i,
$$
for all $X\in\mathcal X$.

The corank of the distribution $\mathcal E\vert_U$ is $r=2n-p$. It
is integrable and has $y_1,\dots,y_r$ as independent integrals.
Indeed, by definition we have
\begin{equation}\label{E_U}
\mathcal E_U=\langle X_f \,\vert \, f= \bar f \circ \pi,\, \bar
f\in C^\infty (\pi(U)) \rangle.
\end{equation}

Since $f=\bar f\circ\pi$ and $y_i$ are $\pi$-vertical we have, in
particular, $\mathcal L_Z f=\mathcal L_Z y_i=0$ (the differential
$df$ and $dy_i$ are semi-basic on $U$). Now, by using $\mathcal
L_{X_{y_i}} f=0$ and Lemma \ref{SB} we get
\begin{equation}\label{integrali}
\mathcal L_{X_f} y_i=0, \qquad i=1,\dots,r.
\end{equation}
The relations \eqref{E_U} and \eqref{integrali} prove the claim.
\end{proof}

\section{Noncommutative contact integrability}

\subsection{} Let us consider a contact vector field $X$ and a {\it contact equation}
\begin{equation}
\dot x=X
 \label{eq}
\end{equation}
on a $(2n+1)$-dimensional contact manifold $(M,\mathcal H)$.

First, recall a general definition of non-Hamiltonian
integrability (e.g., see \cite{Ko, Bo, Zu}), slightly adopted with
respect to the notations above. The equation \eqref{eq} is {\it
(non-Hamiltonian) completely integrable} if there is an open dense
subset $M_{reg}\subset M$ and a proper submersion
\begin{equation}\label{sub*}
\pi: M_{reg}\to W
\end{equation}
to a $p$-dimensional manifold $W$ and an Abelian Lie algebra
$\mathcal X$ of symmetries such that:
\begin{itemize}
\item[(i)] the contact vector field $X$ is tangent to the fibers
of $\pi$; \item[(ii)] the fibers of $\pi$ are orbits of $\mathcal
X$.
\end{itemize}

If \eqref{eq} is completely integrable then $M_{reg}$ is foliated
on $(r+1)$-dimensional tori with a quasi-periodic dynamics. In
nonholonomic mechanics, usually, an additional time
reparametrization is required (e.g., see \cite{Ko, FJ, Jo2}).

However, the above definition does not reflect the underlying
contact structure.

\begin{definition}{\rm
We shall say that the contact equation \eqref{eq} is {\it
noncommutatively contact completely integrable} if, in addition,
$(M_{reg},\mathcal H,\mathcal X)$ is a complete pre-isotropic
contact structure.}\end{definition}

The regularity of the dynamics of integrable contact systems is
described in the following statement.

\begin{theorem}\label{MAIN}
Suppose that the equation \eqref{eq} is noncommutatively contact
completely integrable by means of the submersion \eqref{sub*} and
commuting symmetries $\mathcal X$. Let $F$ be a connected
component of the fiber $\pi^{-1}(w_0)$. Then $F$ is diffeomorphic
to a $r+1$-dimensional torus $\mathbb T^{r+1}$, $r=2n-p$. There
exist an open $\mathcal X$-invariant neighborhood $U$ of $F$, an
$\mathcal X$-invariant contact contact form $\alpha$ on $U$ and a
diffeomorphism $\phi: U\to \mathbb T^{r+1} \times D$,
\begin{equation}\label{action-angle}
\phi(x)=(\theta,y,x)=(\theta_0,\theta_1,\dots,\theta_r,y_1,\dots,y_r,x_1,\dots,x_{2s}),
\quad s=n-r,
\end{equation}
where $D\subset \R^p$ is diffeomorphic to $W_U=\pi(U)$, such that

(i) $\mathcal F\vert_U$ is $\alpha$-complete foliation with
integrals $y_1,\dots,y_r,x_1,\dots,x_{2s}$, while the integrals of
the pseudo-orthogonal foliation $\mathcal E\vert_U=\mathcal
F\vert_U^\perp$ are $y_1,\dots,y_r$.

(ii) $\alpha$ has the following canonical form
\begin{equation}\label{canonical}
\alpha_0=(\phi^{-1})^*\alpha=y_0d\theta_0+y_1d\theta_1+\dots+y_rd\theta_r+g_1dx_1+\dots+g_{2s}
dx_{2s},
\end{equation}
where $y_0$ is a smooth function of $y$ and $g_i$ are functions of
$(y,x)$.

(iii) the flow of $X$ on invariant tori is quasi-periodic
\begin{equation}\label{namotavanje}
(\theta_0,\theta_1,\dots,\theta_r) \longmapsto
(\theta_0+t\omega_0,\theta_1+t\omega_2,\dots,\theta_r+t\omega_r),
\quad t\in\R,
\end{equation}
where frequencies $\omega_0,\dots,\omega_r$ depend only on $y$.
\end{theorem}

\begin{definition}{\rm
We refer to local coordinates $(\theta,y)$ stated in Theorem 4 as
a {\it generalized contact action-angle coordinates}.
}\end{definition}

In the case when the contact manifold is co-oriented ($\mathcal
H=\ker\alpha$) and we have the contact Hamiltonian equation
\eqref{eq_f}, it is convenient to formulate noncommutative
integrability in terms of the first integrals and the Jacobi
bracket as well.

\begin{theorem}
Suppose we have a collection of integrals $f_1,f_2,\dots,f_{2n-r}$
of equation \eqref{eq_f} with the contact Hamiltonian either
$f=f_1$ or $f=1$, where:
\begin{equation}\label{involucija}
[1,f_i]=0, \quad [f_i,f_j]=0, \quad i=1,\dots,2n-r, \quad
j=1,\dots, r.
\end{equation}

Let $F$ be a compact connected component of the level set
$$
\{x\,\vert\,f_1=c_1,\dots,f_{2n-r}=c_{2n-r}\}
$$
and assume
\begin{equation}\label{regular}
df_1 \wedge \dots \wedge df_{2n-r}\ne 0
\end{equation}
on $F$. Then $F$ is diffeomorphic to a $r+1$-dimensional torus
$\mathbb T^{r+1}$. There exist a neighborhood $U$ of $F$ with
local generalized action-angle coordinates \eqref{action-angle} in
which $\alpha$ has the form \eqref{canonical} and the dynamics is
quasi-periodic \eqref{namotavanje}.
\end{theorem}

\begin{proof} Consider the mapping
$$
\pi=(f_1,\dots,f_{2n-r}): \quad M\to\R^{2n-r}.
$$

From \eqref{regular} there exist a neighborhood $U$ of $F$ such
that $\pi\vert_U$ is a proper submersion to $\pi(U)$. Let
$\mathcal F$ be a foliation with leaves that are fibers of $\pi$.
Since $df_i$ are semi-basic 1-forms, \eqref{regular} implies $df_1
\wedge \dots \wedge df_{2n-r}\wedge \alpha \ne 0$. Thus, $\mathcal
F$ is transversal to $\mathcal H\vert_U$ and the infinitesimal
automorphisms of $\alpha$
\begin{equation}\label{simetrije}
Z,X_{f_1},\dots,X_{f_{r}}
\end{equation}
are independent in
$U$.

Further, from \eqref{der} and \eqref{involucija}, we conclude
\begin{eqnarray}
&&[Z, X_{f_i}]=0, \quad [X_{f_i},X_{f_j}]=0, \quad i=1,\dots,2n-r, \quad j=1,\dots,r \nonumber\\
&&\mathcal L_Z f_i=0, \quad L_{X_{f_j}} f_i=0, \quad \mathcal
L_{X_{f_i}} f_j=0, \label{lepe}
\end{eqnarray}

The relations \eqref{lepe} provide that the commuting vector
fields \eqref{simetrije} belong to $\mathcal F$. From the
dimensional reason, they span $\mathcal F$. From \eqref{lepe} we
also get that $f_1,\dots,f_r$ are integrals of the
pseudo-orthogonal distribution $\mathcal E=\mathcal F^\perp$.
Whence $\mathcal E$ is integrable. On the other hand, $\mathcal
F\subset \mathcal E$ implies that the distribution $\mathcal
G=\mathcal F\cap\mathcal H$ is isotropic (item (i) of Theorem 2).

Therefore, $\mathcal F$ is a complete pre-isotropic foliation with
commuting symmetries \eqref{simetrije}. Now, the statement follows
from Theorem 4.
\end{proof}

\begin{proof}[Proof of Theorem \ref{MAIN}]
{\it Step 1 (local bi-fibrations).} Since on each connected
component of the fiber $\pi^{-1}(w_0)$, $\mathcal X$ induces a
transitive action of $\R^{r+1}$ ($r=2n-p$), the connected
components of $\pi^{-1}(w_0)$ are $r+1$-dimensional tori $\mathbb
T^{r+1}$ (e.g., see Arnold \cite{Ar}).

Let us fix some connected component $F$ of $\pi^{-1}(w_0)$.
Consider some $\mathcal X$-invariant connected neighborhood $U$ of
$F$ and a $\mathcal X$-invariant contact form $\alpha$ defining
the distribution $\mathcal H\vert_U=\ker\alpha$ such that the
corresponding Reeb vector field $Z$ belongs to $\mathcal X$ (see
the construction given in Theorem 3).

Let $y'_i=i_{X_i}\alpha$ be contact Hamiltonians of $r$
independent contact vector fields $X_i\in\mathcal X$, $\mathcal
F\vert_U=\langle Z,X_1,\dots,X_r\rangle$. The functions
$y'_1,\dots,y'_r$ are then integrals of the pseudo-orthogonal
foliation  as well (see the proof of Theorem 3). They are
$\pi$-vertical, and by $\bar y'_i$ we denote the corresponding
functions on $W_U=\pi(U)$. Locally, for $U$ small enough, the
foliation $\mathcal E\vert_U$ is also a fibration $\rho_U$ over an
open set $V_U$ diffeomorphic to a ball in $\R^r$ with local
coordinates $\bar y'=(\bar y'_1,\dots,\bar y'_r)$ (the using of
$\bar y'_i$ will be clear from the contexts). Therefore we have a
bi-fibration
$$
\begin{matrix}
 & & U & & \\
 &   \swarrow_{\pi_U}& & _{\rho_U}\searrow  &  \\
 W_U  &&&& V_U
\end{matrix}
$$
with pseudo-orthogonal fibers $\mathcal F\vert_U$ and $\mathcal
E\vert_U$.

Let $\bar x=(\bar x_1,\dots,\bar x_{2s})$ be any collection of
independent functions, where $(\bar y',\bar x)$ are local
coordinates on $W_U$. Let $x_a=\bar x_a\circ \pi_U$,
$a=1,\dots,2s$. By the use of the methods developed by Arnold
\cite{Ar}, it follows that locally we have a trivial toric
fibration $U\cong \mathbb T^{r+1}\times W_U$ with coordinates
$$
(\varphi_0,\dots,\varphi_r,y'_1,\dots,y'_r,x_1,\dots,x_{2s}).
$$
The angular variables $(\varphi_0,\dots,\varphi_r)$ are chosen
such that
$$
Y_\nu={\partial}/{\partial \varphi_\nu}=\sum_{\mu=0}^r
\Lambda_{\nu\mu} X_\mu,
$$
where the Reeb vector field $Z$ of $\alpha$ is denoted by $X_0$
and the invertible matrix $(\Lambda_{\nu\mu})\in GL(r+1)$ depends
only on $(y',x)$.

\medskip

{\it Step 2 (description of $\alpha$).} By construction, the
functions $y'_j=i_{X_j}\alpha$ are $\rho_U$-basic. Since $\mathcal
L_{X_j}\alpha=0$, the 1-forms
$$
i_{X_{j}} d\alpha=-d(\alpha(X_j))=-dy'_j, \qquad j=1,\dots,r
$$
are also $\rho_U$-basic.\footnote{Let $\pi: M\to P$ be a
surjective submersion. A 1-form $\omega$ is {\it semi-basic} if
$i_X\omega=0$ for all vertical vector fields $X$. It is {\it
basic} if $\omega=\pi^*\mu$, where $\mu$ is a 1-form on $P$. In
particular, a basic form is semi-basic as well \cite{LM}.}
Besides, $i_{X_0}d\alpha=i_Zd\alpha=0$. Therefore
\begin{equation}\label{Y}
i_{Y_{\nu}}d\alpha=\sum_{\mu=0}^r \Lambda_{\nu\mu} i_{X_\mu}
d\alpha=-\sum_{\mu=0}^r \Lambda_{\nu\mu}dy'_\mu, \qquad
\nu=0,1,\dots,r
\end{equation}
are $\rho_U$-{semi basic} 1-forms. Here $y'_0\equiv 1$. In
particular, $d\alpha$ does not contain the terms with
$d\varphi_\nu\wedge d\varphi_\mu$. So $\alpha$ takes the form
\begin{equation}\label{local-a}
\alpha=\sum_{\nu=0}^r y_\nu d\varphi_\nu+\sum_{i=1}^r \tilde
f_idy'_i+\sum_{a=1}^{2s}{\tilde g}_a dx_a,
\end{equation}
where $y_\nu=y_\nu(y',x)$, $\nu=0,\dots,r$. Thus, it follows
\begin{equation}\label{YY}
 i_{Y_\nu}
d\alpha=-dy_\nu+\sum_{i=1}^r\frac{\partial \tilde f_i}{\partial
\varphi_\nu}dy'_i +\sum_{a=1}^{2s}\frac{\partial{\tilde
g}_a}{\partial \varphi_\nu}dx_a.
\end{equation}

By combining \eqref{Y}, \eqref{YY} and the fact that the matrix
$(\Lambda_{\nu\mu})$ does not depend on $\varphi$, we obtain that
$\tilde f_i$ and ${\tilde g}_a$ are linear in angular variables.
Since they are periodic in $\varphi_\nu$, they depend only on
$(y',x)$ and
\begin{equation}\label{Y*}
 i_{Y_\nu}
d\alpha=-dy_\nu.
\end{equation}

From \eqref{local-a} and \eqref{Y*} we find the Lie derivatives
\begin{equation*}\label{ne-inv}
\mathcal L_{Y_{\nu}}\alpha=i_{Y_\nu}
d\alpha+di_{Y_\nu}\alpha=-dy_\nu+dy_\nu=0,\qquad \nu=1,\dots,r
\end{equation*}
and conclude that $\alpha$ is invariant with respect to the angle
coordinates vector fields
${\partial}/{\partial\varphi}_\nu=Y_\nu$.

Now, according to Lemma \ref{lema}, the matrix
$(\Lambda_{\nu\mu})$ depends only on $y'$-variables. Therefore,
the 1-forms $i_{Y_\nu}d\alpha$  (see \eqref{Y})  as well as the
functions $y_\nu$ (see \eqref{Y*}) are $\rho_U$-basic. Note that
$y_\nu=i_{Y_\nu}\alpha$ are contact Hamiltonians of the contact
vector fields $Y_\nu$.

Among $y_\nu$ there are $r$ independent functions at every point
in $U$. With eventually shrinking of $U$ and a permutation of
indexes, we can assume that $y_1,\dots,y_r$ are independent and
$y_0=y_0(y_1,\dots,y_r)$ (i.e., $\bar y_1,\dots,\bar y_r$ are new
coordinates on $V_U$). As a result, the contact form reads
\begin{equation}\label{local-a*}
\alpha=\sum_{\nu=0}^r y_\nu d\varphi_\nu+\sum_{i=1}^r
f_i(y,x)dy_i+\sum_{a=1}^{2s} g_a(y,x)dx_a.
\end{equation}

Introducing the new angle variables
\begin{equation}\label{translation}
(\theta_0,\theta_1,\dots,\theta_r)=(\varphi_0,\varphi_1-f_1(y,x),\dots,\varphi_r-f_r(y,x)),
\end{equation}
the form \eqref{local-a*} becomes
$$
\alpha=\sum_{i=0}^r y_id\theta_i+\sum_{a=1}^{2s} g_a(y,x)dx_a +df,
$$
where $f=f(y,x)=\sum_{i=1}^r y_i f_i(y,x)$ is a $\pi_U$-basic
function. Due to the translation \eqref{translation}, the
coordinate vector fields of $\theta$ and $\varphi$ coincide:
${\partial}/{\partial
\theta_\nu}={\partial}/{\partial\varphi_\nu}=Y_\nu$.

\medskip

{\it Step 3 (Moser's deformation, see e.g., \cite{BM, Ge}).} Let
$$
\alpha_0=\sum_{\nu=0}^r y_ \nu d\theta_\nu+\sum_{a=1}^{2s}
g_a(y,x)dx_a.
$$
and $Z=X_0$ be the Reeb vector field of $\alpha$. It is
$\pi_U$-vertical and we have
$$
i_Z \alpha = i_Z \alpha_0=1, \qquad i_Z d\alpha=i_Z d\alpha_0=0,
$$
implying $\mathcal L_Z \alpha=\mathcal L_Z \alpha_0=0$.

Following \cite{BM},  consider the vector field $Y=-fZ$, where $f$
is the $\pi_U$-basic function defined above. The flow $\phi_t$ of
$Y$ is a complete flow that preserves the toric fibration. Define
$\alpha_t=\alpha_0+tdf$. Then we have
$$
\mathcal L_Y \alpha_t =\mathcal L_Y \alpha_0+t\mathcal L_Y
h=\mathcal L_Y
\alpha_0=i_Yd\alpha_0+d(i_Y\alpha_0)=-df=-\partial\alpha_t/\partial
t.
$$

Thus
$$
\frac{d}{dt}(\phi_t^*\alpha_t)=\phi^*_t(\mathcal L_Y
\alpha_t+\frac{\partial\alpha_t}{\partial t})=0,
$$
which implies that $\phi_1^* \alpha_1=\phi_1^*\alpha=\alpha_0$.
Finally,  the required change of variables is  $\phi=\phi_{-1}$.

\medskip

{\it Step 3 (linearization).} Since the system is non-Hamiltonian
completely integrable, we have a quasi-periodic motion on
invariant tori \cite{Bo, Zu}. The special form of a linearization,
where frequencies depend only on $y_1,\dots,y_r$ follows from
Lemma \ref{lema} below. \end{proof}

\begin{remark}\label{action-remark}{\rm
The action functions $y_\nu=i_{Y_\nu}\alpha$ constructed above
have an another interesting interpretation. Let $\gamma_\nu(T)$ be
a cycle homologous to the trajectories of the field
${\partial}/\partial\theta_{\nu}$ restricted to any invariant
torus $T$ within $U$. Then it follows
\begin{equation}\label{actions}
y_\nu\vert_T=\frac{1}{2\pi}\int_{\gamma_\nu(T)} \alpha.
\end{equation}
Indeed, since $d\alpha\vert_T=0$ (the tangent space of $T$ splits
into an isotropic horizontal part and $\R Z=\ker d\alpha$) the
value of the integral \eqref{actions} is the same for all
$\gamma_\nu(T)$ in the same homology class. Then \eqref{actions}
simply follow from \eqref{canonical}. In the opposite direction,
we can use \eqref{actions} as a definition of $y_\nu$. By
construction, the functions $y_\nu$ are $\pi_U$-vertical. As in
the symplectic case (see Nehoroshev \cite{N}), it can be proved
that they are also $\rho_U$-vertical. }\end{remark}

\begin{remark}\label{REEB}{\rm Let
$ Z=z_0(y)Y_0+\dots+z_r(y)Y_r$ be the local expression of the Reeb
vector field. It is uniquely determined from the conditions
$i_Z\alpha_0=1$, $i_Z d\alpha_0=0$, i.e.,
\begin{equation}\label{z}
z_0y_0+\dots+z_ry_r=1, \qquad z_0 dy_0+\dots+z_r dy_r=0.
\end{equation}

If $z_0=0$ at some point $y=\tilde y$, then $z_1 dy_1+\dots+z_r
dy_r=0$ at $\tilde y$. Since $dy_i$, $i=1,\dots,r$ are independent
1-forms, we get $z_1=\dots=z_r=0$ at $\tilde y$ which contradict
\eqref{z}. Therefore $z_0\ne 0$ on $V_U$. Now, by solving
\eqref{z} we get
$$
z_0=\frac{1}{y_1\frac{\partial y_0}{\partial
y_1}+\dots+y_r\frac{\partial y_0}{\partial y_r}-y_0}, \quad
z_i=-\frac{1}{z_0}\frac{\partial y_0}{\partial y_i}, \qquad
i=1,\dots,r.
$$
Therefore, typically, the flow of the Reeb vector field is
quasi-periodic and everywhere dense in invariant tori. Also,
typically, the induced pseudo-isotropic foliation $\mathcal
G=\mathcal F \cap \mathcal H$ has noncompact invariant manifolds.
}\end{remark}

\begin{remark}{\rm
Consider the 1-form $\gamma=\sum_{a=1}^{2s}
g_a(y,x)dx_a=\alpha_0-\sum_{\nu=0}^r y_\nu d\theta_\nu$. Since
$d\alpha_0$ has the maximal rank, according to Darboux's theorem
\cite{LM}, there is a coordinate transformation $q_j=q_j(y,x),
p_j=p_j(y,x)$, $j=1,\dots,s$ such that $\gamma=p_1 dq_1+\dots+p_s
dq_s$, i.e.,
$$
\alpha_0=y_0d\theta_0+y_1d\theta_1+\dots+y_rd\theta_r+p_1dq_1+\dots+p_s
dq_s.
$$
}\end{remark}

\begin{lemma}\label{lema}
Let $(M,\mathcal H,\mathcal X)$ be a complete pre-isotropic
contact structure and  let $U\subset M$ be an $\mathcal
X$-invariant set endowed with an $\mathcal X$-invariant contact
form $\alpha$. Suppose

(i) The foliation $\mathcal F\vert_U=\ker\pi_*\vert_U$ is
$\alpha$-complete and there exist everywhere independent integrals
$y_1,\dots,y_r: U\to\R$, of the pseudo-orthogonal foliation
$\mathcal E\vert_U=\mathcal F\vert_U^\perp$.

(ii) Let $X$ be a contact vector field tangent to the fibers of
$\pi_U$, commuting with $\mathcal X$.

Then $X$ can be written as a fiber-wise linear combination
$$
X=f_0 Z+f_1 X_1+\dots+ f_r X_r,
$$
where functions $f_0,\dots,f_r$ depend only on $y$, $Z$ is the
Reeb vector field of $\alpha$ and $X_i=X_{y_i}$ are contact
Hamiltonian vector fields of $y_i$, $i=1,\dots,r$.\end{lemma}

\begin{proof} Under the assumption (i), $Z,X_1,\dots,X_r$ are independent
vector fields that generate $\alpha$-complete pre-isotropic
foliation $\mathcal F\vert_U$.

Next, we shall prove that $X$ commute with $Z$. Firstly, notice
that $Z$ commute with $\mathcal X$.\footnote{Here we consider
slightly more general situation then it is needed for Theorem 4,
where, by construction of $\alpha$, $Z$ is already an element of
$\mathcal X$. However, we shall use the above formulation for a
proof of Proposition \ref{reeb}.} Indeed, let $Y\in\mathcal X$. We
have
\begin{equation}\label{pomoc3}
\Phi([Y,Z])=i_{[Y,Z]}\alpha=\mathcal L_Y i_{Z} \alpha - i_{Z}
\mathcal L_Y \alpha=0.
\end{equation}
Since \eqref{iso} is an isomorphism we get $[Y,Z]=0$.

Secondly, note that any $\pi$-vertical vector field $K$ (not need
to be contact field) that commute with $\mathcal X$, commute with
$X$ as well. Indeed, any point in $U$ has a $\pi$-invariant
neighborhood $U'$ where $K$ can be written as a linear combination
$\sum_{\nu=1}^r g_\nu Y_\nu$ where $g_\nu$ are $\pi$-basic
functions and $(Y_0,\dots,Y_r)$ is a collection of vector fields
in  $\mathcal X$ that generate $\mathcal F\vert_U$. Therefore
$$
[X,K]=\sum_{\nu=0}^r [X,g_\nu
Y_\nu]=\sum_{\nu=0}^r(g_\nu[X,Y_\nu]+dg_\nu(X) Y_\nu)=0.
$$

From the above considerations it follows that $X$ commute with
$Z$. Let $f=i_X \alpha$ be the contact Hamiltonian of $X$. Since
$[Z,X_f]=0$ we have $[1,f]=0$ and $df$ is a semi-basic form. Since
$X$ is $\pi_U$-vertical, we have $\mathcal L_{X_f} g=0$, where $g$
is any local integral of $\mathcal F$. It is clear that $dg$ is
semi-basic and applying Lemma \ref{SB} again, it follows $\mathcal
L_{X_g} f=0$. Whence $f$ is an integral of the pseudo-orthogonal
foliation $\mathcal E\vert_U$.

Under the assumptions of Lemma \ref{lema}, integrals of $\mathcal
E\vert_U$ are functions of $y$ and we have $f=f(y)$. Let
$f_i={\partial f}/{\partial y_i}$, $i=1,\dots,r$. The forms
$df,dy_1,\dots,dy_r$ are semi-basic, so
\begin{eqnarray*}
X &=& \Phi^{-1}(f)= fZ+\alpha^\sharp(df)=f Z+\sum_{i=1}^r
f_i\alpha^{\sharp}(dy_i)\\
&=& f Z + \sum_{i=1}^r f_i(X_i-y_i Z)= f_0 Z+f_1 X_1+\dots+ f_r
X_r,
\end{eqnarray*}
where $f_0=f-(y_1f_1+\dots+y_rf_r)$. \end{proof}

\begin{remark}{\rm
Let $X$ be $\pi_U$-horizontal contact vector field. From the proof
of the lemma, we see that commuting of $X$ with $\mathcal X$ is
equivalent to the commuting with the Reeb vector field $Z$, i.e,
with the condition that $X$ is an infinitesimal automorphisms of
$\alpha$. Also, the condition that $\mathcal F\vert_U$ is
$\alpha$-complete is equivalent to the condition that  $Z$ is a
section of $\mathcal F\vert_U$, see Proposition \ref{reeb} given
below. }\end{remark}

\subsection{Discrete systems}

Khesin and Tabachnikov  defined integrability of discrete
\begin{equation}\label{discrete}
\Psi: M\to M,
\end{equation}
and continuous contact systems \eqref{eq} in terms of the
existence of an invariant complete pre-Legendrian foliation
$\mathcal F$, with additional property that on every leaf $F$ of
$\mathcal F$, the foliation $\mathcal G\vert_F$ has a holonomy
invariant transverse smooth measure. It turns out that this
condition implies the existence of a global contact form $\alpha$
and that $\mathcal G$ is an $\alpha$-complete Legendrian foliation
\cite{KT}.

As in \cite{KT}, we can say that a discrete contact system
\eqref{discrete} that preserves the contact form $\alpha$ is {\it
integrable in a noncommutative sense} if it possesses an
$\alpha$-complete pre-isotropic invariant foliation $\mathcal F$.
Also, following the lines of the proof of  Lemma 3.5 \cite{KT},
one can prove that $\alpha$ determines a holonomy invariant
transverse smooth measure of the foliation $\mathcal G=\mathcal
F\cap \mathcal H$ restricted to the leaves of $\mathcal F$.

\subsection{Examples}

For $s=0$, Theorem 4 recover contact action-angle coordinates
given by Banyaga and Molino \cite{BM}. If $M$ is a compact
manifold with a regular effective contact action of $\mathbb
T^{n+1}$, then $W$ is the sphere $S^{n}$ and for $n \ge 3$, $M$ is
diffeomorphic to $\mathbb T^{n+1}\times S^{n}$ (see Lutz
\cite{Lu}).

Besides noncommutatively integrable geodesic flow restricted to
the unit co-sphere bundles \cite{BJ, Jo1}, a natural class of
examples of contact flows integrable in a noncommutative sense are
the Reeb flows on $K$-contact manifolds $(M^{2n+1},\alpha)$ where
the rank of the manifold is less then $n+1$ (see Yamazaki
\cite{Y1} and Lerman \cite{Le2}).

The regular and almost regular contact manifolds studied by
Boothby and Wang \cite{BW} and Thomas \cite{Th} provide the most
degenerate examples with $\dim W=\dim M-1$.

The billiard system within an ellipsoid in the Euclidean space
$\R^n$ is one of the basic examples of integrable mappings (e.g.,
see \cite{Ve, DrRa}). Similarly, the billiard system inside an
ellipsoid in the pseudo-Euclidean space $\mathbb R^{k,n-k}$ is
completely integrable as well. Here, the billiard system is
described by a symplectic transformation on the spaces of
space-like and time-like geodesics, while it is a contact
transformation on the space of light-like geodesics (for more
details, see Khesin and Tabachnikov \cite{KT0, KT}). The
considered billiard systems are defined within ellipsoids with
different semi-axis. Further properties of ellipsoidal billiards
in the pseudo-Euclidean spaces have been studied in \cite{DR2},
where description of periodical trajectories has been derived,
including the cases of symmetric ellipsoids. It can be proved that
the billiard systems, both in $\R^n$ and $\R^{k,n-k}$, within
symmetric ellipsoids are completely integrable in the
noncommutative sense (the geodesic flow on a symmetric ellipsoid
is considered in \cite{BD}). In particular, the billiard maps
restricted to the space of null geodesics are noncommutatively
completely integrable contact transformations.

\section{Complete pre-isotropic structures of the Reeb type}

\subsection{}
In this section we consider some global properties of the
fibration \eqref{submersion}.

\begin{proposition} Let $(M,\mathcal H,\mathcal X)$ be a
complete pre-isotropic contact structure and assume that $\mathcal
H$ is co-oriented. Then there exist a global contact form $\alpha$
representing $\mathcal H$ and invariant by elements of $\mathcal
X$.
\end{proposition}

\begin{proof}  We can cover $W$ by open sets $W_i$ such that
we have contact 1-forms $\alpha_{U_i}$ invariant by $\mathcal X$
 on every
$U_i=\pi^{-1}(W_i)$ (Theorem 3). Let $\bar \lambda_i$ be the
partition of unity subordinate to covering $\{W_i\}$. Since
$\mathcal H$ is oriented, for all nonempty intersections $U_i\cap
U_j$, we have smooth positive functions $a_{ij}$,
$\alpha_{U_i}=f_{ij} \alpha_{U_j}\vert_{U_i\cap U_j}$.

Define the 1-form $\alpha$ by $ \alpha=\sum_i \lambda_i
\alpha_{U_i}$, $\lambda_i=\bar\lambda_i\circ\pi. $ Then, on $U_k$
we have
\begin{equation}\label{a-k}
\alpha=a_k\alpha_{U_k},
\end{equation}
where $a_k=\sum_{i, U_i\cap U_k \ne \emptyset} \lambda_i f_{ki}>0$
is a $\pi$-basic function. Whence $\alpha$ is a contact form that
define $\mathcal H$.

It remains to prove $\mathcal X$-invariance of $\alpha$. Let $X\in
\mathcal X$. Then $\mathcal L_X \lambda_i=0$. Further, by
construction, $X$ preserve all local contact forms $\alpha_{U_i}$.
Thus
$$
\mathcal L_X \alpha=\sum_i(\mathcal L_X
\lambda_i)\alpha_{U_i}+\lambda_i\mathcal L_X\alpha_{U_i}=0.
$$
\end{proof}

Let $Z$ be the Reeb vector field of the globally $\mathcal
X$-invariant contact form $\alpha$. Then, as in \eqref{pomoc3}, we
get $[Z,Y]=0$, $Y\in\mathcal X$. However, it turns out that the
foliation $\mathcal F=\ker\pi_*$ not need to be $\alpha$-complete
since $Z$ not need be a section of $\mathcal F$.

Recall that a contact toric action on a co-oriented contact
manifold $(M,\alpha)$ is of the {\it Reeb type} if the Reeb vector
field corresponds to an element of the Lie algebra of the torus
\cite{BG}. Similarly, we give the following definition.

\begin{definition}{\rm
Let $(M,\alpha)$ be a co-oriented contact manifold with a complete
pre-isotropic contact structure defined by commuting infinitesimal
automorphisms $\mathcal X$ of $\alpha$, such that the associated
Reeb vector field $Z$ is a section of $\mathcal F=\ker\pi_*$. We
refer to a triple $(M,\alpha,\mathcal X)$ with the above property
as a {\it complete pre-isotropic structure of the Reeb
type.}}\end{definition}

\begin{proposition}\label{reeb}
Let  $(M,\alpha,\mathcal X)$ be a complete pre-isotropic structure
of the Reeb type. Then the associated foliation $\mathcal
F=\ker\pi_*$ is $\alpha$-complete.
\end{proposition}

\begin{proof} Locally, every leaf $F$ of $\mathcal F$ has a
$\pi$-invariant neighborhood $U$ with local generalized contact
action-angle coordinates \eqref{action-angle} in which $\mathcal
H$ is represented by the contact form $\alpha_0=\sum_\nu y_\nu
d\theta_\nu+\sum_a g(y,x)dx_a$ and $\mathcal F\vert_U$ is
$\alpha_0$-complete (Theorem 4). We need to prove that $\mathcal
F$ is complete with respect to the contact form $\alpha$ as well.

We have $\alpha\vert_U=\frac{1}{a}\cdot\alpha_0$ for some
nonvanishing function $a: U\to \R$. In what follows, by
$Z^\alpha$, $Z^{\alpha_0}$, $X^\alpha_f$, $X^{\alpha_0}_f$ and
$\Phi_\alpha$, $\Phi_{\alpha_0}$ we denote the Reeb vector fields,
contact Hamiltonian vector fields and the isomorphisms \eqref{iso}
with respect to $\alpha$ and $\alpha_0$, respectively. They are
related by
$$
X^\alpha_f=\Phi^{-1}_\alpha(f)=\Phi^{-1}_{\alpha_0}(af)=
X^{\alpha_0}_{af}, \qquad
Z^\alpha=\Phi_\alpha^{-1}(1)=\Phi_{\alpha_0}^{-1}(a)=X^{\alpha_0}_a
$$
(see Proposition 13.7, \cite{LM}).

On the other hand, by the argument used in \eqref{pomoc3}, with
$\Phi$ replaced by $\Phi_\alpha$, we get $[Z^\alpha,X]=0$,
$X\in\mathcal X$. Therefore we can apply Lemma \ref{lema} with
$Z^\alpha=X_a^{\alpha_0}$ and $\alpha_0$, instead of $X$ and
$\alpha$, concluding that $a$ is a function of actions variables
$y=(y_1,\dots,y_r)$ only.

Let $f$ be an integral of $\mathcal F$. Since $da$ and $df$ are
semi-basic, we get that the contact Hamiltonian vector field
\begin{eqnarray*}
X^\alpha_f &=& \Phi^{-1}_\alpha(f)=\Phi^{-1}_{\alpha_0}(af)\\
&=& (af)Z^{\alpha_0}+\alpha_0^\sharp({adf+fda})=(af)
Z^{\alpha_0}+a\alpha_0^\sharp(df)+f\alpha_0^\sharp(da)\\
&=&af Z^{\alpha_0} +
a(X^{\alpha_0}_f-fZ^{\alpha_0})+f(X^{\alpha_0}_a-aZ^{\alpha_0})\\
&=& aX^{\alpha_0}_f+fZ^{\alpha}-af Z^{\alpha_0},
\end{eqnarray*}
is a section of pseudo-orthogonal complement of $\mathcal F$ with
respect to $\alpha_0$. Thus, the pseudo-orthogonal complements of
$\mathcal F$ with respect to $\alpha$ and $\alpha_0$ coincides.
This completeness the proof. \end{proof}

\begin{remark}{\rm
Let us return to the construction of an invariant contact form
$\alpha$ given in Proposition 1. From the proof of Proposition
\ref{reeb}, we obtain that $\mathcal F=\ker \pi_*$ is
$\alpha$-complete if the functions $a_k$ defined by \eqref{a-k}
depend only on actions variables. If this is not the case, suppose
additionally that the Reeb vector field $Z$ is transversal to
$\mathcal F$ at every point. Then we can consider the foliation
$\tilde{\mathcal F}$ generated by $\mathcal X$ and $Z$. It can be
proved that $\tilde{\mathcal F}$ is $\alpha$-complete. Note that
if $n=p$, i.e, $(M,\mathcal H,\mathcal X)$ is a {regular
completely integrable contact structure}, then $a_k$ depends only
on action variables and $\mathcal F$ is $\alpha$-complete.
}\end{remark}

\subsection{}
Let $(M,\alpha,\mathcal X)$ be a complete pre-isotropic structure
of the Reeb type and assume the fibers of  \eqref{submersion} are
connected. Theorem 4 and Proposition 2 provide that $\pi: M\to W$
is a toric fibration. There is an open covering $W_i$ of $W$ and
local trivializations $\phi_i: U_i=\pi^{-1}(W_i)\to \mathbb
T^{r+1} \times D_i$,
\begin{equation*}\label{action-angle*}
\phi_i(x)=(\theta^i,y^i,x^i)=(\theta^i_0,\theta^i_1,\dots,\theta^i_r,y^i_1,\dots,y^i_r,x^i_1,\dots,x^i_{2s}),
\quad s=n-r,
\end{equation*}
where $D_i\subset \R^p$ is an open set diffeomorphic to $W_i$,
such that
\begin{itemize}
\item[(i)] the fibers of $\pi$ are represented as the level sets
of functions $(y^i,x^i)$, where the action variables $y^i$ are
integrals of the pseudo-orthogonal foliation $\mathcal E=\mathcal
F^\perp$ restricted to $U_i$;

\item[(ii)] $\alpha$ has the following canonical form
\begin{equation*}\label{canonical*}
\alpha_i=(\phi^{-1}_i)^*\alpha=y^i_0d\theta^i_0+y^i_1d\theta^i_1+\dots+y^i_rd\theta^i_r+g^i_1dx^i_1+\dots+g^i_{2s}
dx_{2s}^i,
\end{equation*}
where $y^i_0$ is a smooth function of $y^i$ and $g_a^i$ are
functions of $(y^i,x^i)$.
\end{itemize}

\begin{proposition}
Suppose that the intersection of $W_i$ and $W_j$, i.e., of $U_i$
and $U_j$ is connected. Then on $U_i \cap U_j$ we have the
following transition formulas:
\begin{eqnarray}
&& \theta^j_\nu=\sum_{\mu=0}^r M^{ij}_{\nu\mu}
(\theta_\mu^i+F^{ij}_\mu(y^i,x^i)),\label{prevodjenje-uglovi}\\
&& y^j_\nu=\sum_{\mu=0}^r K^{ij}_{\nu\mu}
y^i_\mu,\label{prevodjenje-y}\qquad
\nu=0,\dots,r,\\
&& x^j_a=X^{ij}_a(y^i,x^i), \qquad
a=1,\dots,2s,\label{prevodjenje-x}
\end{eqnarray} where matrixes
$K^{ij}=(K^{ij}_{\nu\mu})$ and $M^{ij}=(K^{ij}_{\nu\mu})$ belong
to $GL(r+1,\mathbb Z)$, $M=(K^T)^{-1}$, and functions
$X^{ij}_a(y^i,x^i)$, $F^{ij}_\nu(y^i,x^i)$ satisfy
\begin{equation}\label{prevodjenje-F}
g_a^i=\sum_{b=1}^{2s} g^j_b\frac{\partial X^{ij}_b}{\partial
x^i_a},\qquad \sum_{b=1}^{2s} g^j_b\frac{\partial
X^{ij}_b}{\partial y^i_k}+\sum_{\nu=0}^r y^i_\nu \frac{\partial
F^{ij}_\nu}{\partial y^i_k}=0.
\end{equation}
\end{proposition}

\begin{proof} Since $y^i$ and $y^j$ (respectively, $(y^i,x^i)$ and
$(y^j,x^j)$) are integrals of the pseudo-orthogonal foliation
$\mathcal E$ (respectively, of $\mathcal F$) we have:
\begin{equation}\label{pomoc1}
\theta^j_\nu=\Theta^{ij}_\nu(\theta^i,y^i,x^i), \qquad
y^j_k=Y^{ij}_k(y^i), \qquad x^j_a=X^{ij}_a(y^i,x^i),
\end{equation}
$\nu=0,\dots,r, k=1,\dots,r, a=1.\dots,2s$.

Let us fix some invariant torus $T=\pi^{-1}(w_0)$ within $U_i \cap
U_j$ $(w_0\in W_i \cap W_j)$. From \eqref{actions}, we have
$$
y^j_\nu\vert_T=\int_{\gamma^j_\nu(T)} \alpha = \sum_{\mu=0}^r
K^{ij}_{\nu\mu} \int_{\gamma^i_\mu(T)}\alpha=\sum_{\mu=0}^r
K^{ij}_{\nu\mu} y^i_\mu\vert_T,
$$
where $K^{ij}\in GL(r+1,\mathbb Z)$ is a matrix which relates two
different bases of cycles $(\gamma_0^j(T),\dots,\gamma_r^j(T))$
and $(\gamma_0^i(T),\dots,\gamma_r^i(T))$ defined in Remark
\ref{action-remark}. From \eqref{pomoc1} and the connectedness of
$W_i \cap W_j$ the matrix $K^{ij}$ is constant. This proves
\eqref{prevodjenje-y}. Therefore
 $$
 i_{{\partial}/{\partial\theta^j_\nu}}d\alpha=-dy^j_\nu=-\sum_\mu
K^{ij}_{\nu\mu}dy^i_\mu=\sum_\mu K^{ij}_{\nu\mu}
i_{{\partial}/{\partial\theta^i_\mu}}d\alpha,
$$
implying that ${\partial}/{\partial\theta^j_\nu}-\sum_\mu
K^{ij}_{\nu\mu}{\partial}/{\partial\theta^i_\mu}\in\ker d\alpha=\R
Z$.

Let $\lambda Z$ be the difference of
${\partial}/{\partial\theta^j_\nu}$ and $\sum_\mu
K^{ij}_{\nu\mu}{\partial}/{\partial\theta^i_\mu}$. Then
$$
\lambda=\alpha (\lambda
Z)=\alpha({\partial}/{\partial\theta^j_\nu}-\sum_\mu
K^{ij}_{\nu\mu}{\partial}/{\partial\theta^i_\mu})=y_\nu^j-\sum_\mu
K^{ij}_{\nu\mu}y_\mu^i=0.
$$

Thus, from \eqref{pomoc1}, permuting the indexes $i$ and $j$, we
obtain
$$
\frac{\partial}{\partial\theta^j_\nu}=\sum_\mu\frac{\partial
\Theta^{ji}_\mu}{\partial\theta^j_\nu}\frac{\partial}{\partial\theta^i_\mu}=
\sum_\mu K^{ij}_{\nu\mu}\frac{\partial}{\partial\theta^i_\mu},
$$
leading to the fact that $\Theta^{ji}_\mu$ is linear in
$\theta^j_\nu$ and that can be written into a form
$$
\Theta^{ji}_\mu=\sum_\nu
\left(K^{ij}_{\nu\mu}\theta^j_\nu+F^{ji}_\nu(y^j,x^j)\right).
$$
From the above expression we get \eqref{prevodjenje-uglovi}, where
$\sum_{\lambda=0}^r
K_{\lambda\mu}^{ij}M_{\lambda\nu}^{ij}=\delta_{\nu\mu}$.

Replacing \eqref{prevodjenje-y} and the differentials of
\eqref{prevodjenje-uglovi}, \eqref{prevodjenje-x} into the
identity
\begin{equation}\label{pomoc2}
\sum_{\nu=0}^r y^i_\nu d\theta^i_\nu+\sum_{a=1}^{2s}
g^i_a(y^i,x^i)dx^i_a= \sum_{\lambda=0}^r y^j_\lambda
d\theta^j_\lambda+\sum_{b=1}^{2s} g^j_b(y^j,x^j)dx^j_b,
\end{equation}
and compering the terms with $dx^i_a$ and $dy^i_k$ we get
\eqref{prevodjenje-F}. \end{proof}

\subsection{}
The study of toric fibrations within the symplectic geometry
framework is based on the papers of Duistermaat \cite{Du}
(Lagrangian fibration) and Dazord and Delzant \cite{DD} (isotropic
fibrations). On the other side, Banyaga and Molino defined
characteristic invariants of
 {\it regular} and {\it singular} completely integrable contact structures and
 proved a classification theorem: two completely integrable contact structures
 with the same invariants are isomorphic \cite{BM}.
For contact toric actions and singular completely integrable
contact structures, see also \cite{BG, Le} and \cite{Mi},
respectively.

Here we consider the existence of global contact action-angles
coordinates by using the arguments already used in the paper.

The possibility of taking all matrices $K^{ij}$ and $M^{ij}$ equal
to the identity reflects the fact that the fibration by the
invariant tori is a principal $\mathbb T^{r+1}$-bundle. When this
does not happen, it is said that we have nontrivial monodromy
\cite{Du}.

Let $W'\subset W$, $\dim W'=\dim W$ be a connected compact
submanifold (with a smooth boundary) and consider the fibration
$\pi: M'\to W'$, $M'=\pi^{-1}(W')$. It is obvious that the
necessary condition for the existence of global contact
action-angles variables is that $M'\to W'$ is a trivial principal
bundle.

The following sufficient, but not necessary, conditions for $M'\to
W'$ to be trivial are well known  (e.g., see \cite{FS}):

\begin{itemize}
\item[(i)] If $W'$ is simply connected then $\pi: M'\to W'$ is a
principal $\mathbb T^{r+1}$ bundle.

\item[(ii)] In addition, if the second cohomology group
$H^2(W',\mathbb Z)$ vanish then the principal bundle is trivial
and $M'$ is diffeomorphic to $\mathbb T^{r+1}\times W'$.
\end{itemize}

Indeed, if $W'$ is simply connected then the monodromy of the
restricted fibration $\pi: M'\to W'$ is trivial providing that
$\pi: M'\to W'$ is a principal $\mathbb T^{r+1}$ bundle. For the
second assertion, note that the Chern class of $\mathbb
T^{r+1}=U(1)\times\dots U(1)$-bundle is equal to
$$
c=c(L_0\oplus\dots \oplus L_r)=(1+c_1(L_0))\dots (1+c_1(L_r))
$$
where $L_\nu$ is the bundle associated to the $\nu$-th factor
$U(1)$. They are all trivial in the case $H^2(M,\mathbb Z)=0$.
Whence, $\mathbb T^{r+1}$-bundle is also trivial.

Now we can formulate the following statement.

\begin{theorem}[Global contact action-angles variables]
Let $(M,\mathcal \alpha,\mathcal X)$ be a complete pre-isotropic
structure of the Reeb type and  let $W'\subset W$, $\dim W'=\dim
W$ be a connected compact submanifold (with a smooth boundary)
such that

(i) $\pi: M'\to W'$ is a trivial principal $\mathbb T^{r+1}$
bundle, $M'=\pi^{-1}(W')$.

(ii) There exist everywhere independent functions $\bar
x_1,\dots,\bar x_{2s}$ defined is some neighborhood of $W'$
satisfying:
\begin{equation}\label{nezavisnost}
\langle dx_1,\dots,dx_{2s}\rangle \cap \mathcal E^0=0,
\end{equation}
where $x_a=\bar x_a\circ \pi$ and $\mathcal E=\mathcal F^\perp$ is
the pseudo-orthogonal foliation of $\mathcal F$.

Then there exist global action-angle variables
$(\theta_0,\dots,\theta_r,y_0,\dots,y_r)$ and functions $\bar
g_1,\dots, \bar g_{2s}: W'\to\R$ such that the contact form
$\alpha$ on $M'$ reads
\begin{equation}\label{global}
\alpha_0=y_0d\theta_0+\dots+y_r d\theta_r+\pi^*(\bar g_1d\bar
x_1+\dots+\bar g_{2s}d\bar x_{2s}).
\end{equation}
\end{theorem}

\begin{remark}{\rm  Proposition 3 and Theorem 6 are contact analogues of Proposition 1 and Theorem 2' in
Nehoroshev \cite{N}, respectively. In Theorem 2' \cite{N}, instead
of the condition (i), the condition that $W'$ is a
simply-connected manifold with vanishing of the second cohomology
class $H^2(W',\mathbb R)$ is used. A variant of the statement with
noncompact invariant manifolds is proved in
\cite{FS}.}\end{remark}

\begin{proof} Since $\pi: M'\to W'$ is a trivial principal $\mathbb
T^{r+1}$ bundle, there exist global angles variables
$(\varphi_1,\dots,\varphi_r)$. Repeating the arguments used in the
proof of Theorem 4, we get that the coordinate vector fields
$Y_\nu={\partial}/{\partial \varphi_\nu}$ preserve $\alpha$ and we
can define actions as their contact Hamiltonians:
$$
y_{\nu}=\Phi(Y_\nu)=i_{Y_\nu}\alpha: \quad M'\to \mathbb R, \qquad
\nu=0,\dots,r.
$$

They are redundant integrals of the pseudo-orthogonal foliation
that satisfy relations \eqref{z}, where $z_\nu$ are the components
of the Reeb vector field $Z$ with respect to vector fields
$Y_\nu$. The functions $y_\nu$ are $\pi$-basic and let $\bar
y_\nu$ be the corresponding functions on $W'$, $y_\nu=\bar y_\nu
\circ\pi$. They are subjected to the constrains
\begin{equation}\label{z*}
\bar z_0 \bar y_0+\dots+\bar z_r y_r=1, \qquad \bar z_0 d\bar
y_0+\dots+\bar z_r d\bar y_y=0,
\end{equation}
where $z_\nu=\pi\circ\bar z_\nu$.

Moreover, according to the assumption \eqref{nezavisnost}, in a
neighborhood of any point $w_0\in W'$, we can take $r$ independent
functions among $\bar y_\nu$ that are independent of $\bar
x_1,\dots,\bar x_{2s}$ providing a local coordinate chart.

Let $\{W_i\}$ be a finite covering of $W'$ such that on every
$W_i$ we can take local coordinates $(\bar y^i,\bar x)$, where
$(\bar y^i_1,\dots,\bar y^i_r)$ is a subcollection of redundant
actions $(\bar y_0,\dots,\bar y_r)$.

As in Theorem 4 we get that the contact form in
$U_i=\pi^{-1}(W_i)$ reads
\begin{equation}\label{semi-global}
\alpha^i=\alpha_\theta+\pi^*\alpha_y^i+\pi^*\alpha_x^i
\end{equation}
where $ \alpha_\theta=\sum_{\nu=0}^r y_\nu d\varphi_\nu$, $
\alpha^i_y=\sum_{k=1}^r \bar f^i_k(\bar y^i,\bar x)d\bar y^i_k$,
$\alpha^i_x=\sum_{a=1}^{2s} \bar g^i_a(\bar y^i,\bar x)d\bar
x_{a}. $

Thus, on $M'$ we have a unique decomposition $
\alpha=\alpha_\theta+\pi^*\alpha_y+\pi^*\alpha_x$, locally given
by \eqref{semi-global}. It is obvious that we can write $\alpha_x$
as $ \alpha_x=\sum_{a=1}^{2s} \bar g_a d \bar x_{a}, $ where $\bar
g_a: W'\to \R$.

Next, consider the filtration
$$
V_1=W_1 \subset V_2=W_1 \cup W_2 \subset \dots \subset V_N=W_1
\cup\dots\cup W_N=W'.
$$

Applying  Lemma \ref{LEMA} given below $(N-1)$ times we obtain
functions $\bar f_0, \dots, \bar f_r: W'\to\R$, satisfying the
identities
$$
\bar f_0d\bar y_0+\bar f_1 d\bar y_1+\dots+\bar f_rd\bar y_r=\bar
f^i_1d\bar y^i_1+\dots \bar f_rd\bar y^i_r
$$
on every $W_i$.

Therefore, after globally defined transformation
\begin{equation*}\label{translation*}
(\theta_0,\theta_1,\dots,\theta_r)=(\varphi_0-f_0,\varphi_1-f_1,\dots,\varphi_r-f_r),
\qquad f_\nu=\bar f_\nu\circ \pi,
\end{equation*}
the form $\alpha$ becomes
\begin{equation*}\label{semi-global*}
\alpha=\sum_{\nu=0}^r y_\nu d\theta_\nu+\pi^*\sum_{a=1}^{2s}\bar
g_a d \bar x_{a}+df,
\end{equation*}
where $f=\sum_\nu y_\nu f_\nu$ is a $\pi$-basic function. Now, as
in Theorem 4, applying Moser's deformation for a compact manifold
$M'$ and family of forms $\alpha_t=\alpha_0+tdf$ we get the
required statement. \end{proof}

\begin{lemma}\label{LEMA} Suppose that on  $W'$ we have
an open set $U$ with local coordinates $(\bar y_1,\dots,\bar y_r)$
and an open set $V$ endowed with $1$-forms
$$
\gamma_U=F_1 d\bar y_1+\dots+F_\nu d\bar y_r, \qquad \gamma_V=G_1
d\bar y_0+\dots+G_\nu d\bar y_r,
$$
that ere equal on the intersection $U\cap V$. Then there exist
functions $E_0,\dots,E_r$ defined on $U \cup V$ satisfying
$$
\gamma_U=E_0 d\bar y_0+\dots+E_r d\bar y_r\vert_U, \qquad
\gamma_V=E_0 d\bar y_0+\dots+E_r d\bar y_r\vert_V.
$$
\end{lemma}

\begin{proof} The statement is trivial if $U\cap V=\emptyset$.
Assume $U\cap V \ne \emptyset$. According to the constraints
\eqref{z*}, the form $\gamma_U$ does not change under the addition
of terms proportional to $\bar z_0 d\bar y_0+\dots+\bar z_r d\bar
y_y$. We are looking for a function $A: U\to \mathbb R$ that
satisfies
\begin{equation}\label{jednacine}
A\bar z_0=G_0,\quad  F_1+A\bar z_1=G_1, \quad \dots \quad
F_r+A\bar z_r=G_r
\end{equation}
on $U\cap V$. Although it is overdetermined system, due to the
condition that $\gamma_U=\gamma_V$ it has an unique solution.
Indeed, on $U$ we have $\bar z_0\ne 0$ and ${\partial \bar
y_0}/{\partial \bar y_i}=-\bar z_i/\bar z_0$ (see Remark
\ref{REEB}). Therefore, the equality $\gamma_U=\gamma_V$ implies
the following compatibility conditions
\begin{equation}\label{saglasnost}
F_1=-G_0\frac{\bar z_1}{\bar z_0}+G_1,\dots, F_r=-G_0\frac{\bar
z_r}{\bar z_0}+G_r, \qquad \bar y\in U\cap V.
\end{equation}

From \eqref{saglasnost}, we obtain that $A=G_0/\bar z_0: U \cap V
\to \mathbb R$ is a  solution of \eqref{jednacine}. Now we take an
arbitrary extension of  $A$ from $U\cap V$ to $U$ and define
$$
E_0\vert U=A, \, E_0\vert_V=G_0, \quad E_i=F_i+A\bar z_i\vert_U,\,
E_i=G_i\vert_V, \quad i=1,\dots,r.
$$
\end{proof}

\subsection*{Acknowledgments}  This research was supported by the
Serbian Ministry of Science Project 174020, Geometry and Topology
of Manifolds, Classical Mechanics and Integrable Dynamical
Systems. The part of the paper is written during author's visiting
SISSA in July 2010. Author would like to thanks Professor Dubrovin
for kind hospitality.

\end{document}